\documentclass[11pt]{amsart}
\usepackage{amsmath}
\usepackage[active]{srcltx}
\usepackage{t1enc}
\usepackage[latin2]{inputenc}
\usepackage{verbatim}
\usepackage{amsmath,amsfonts,amssymb,amsthm}
\usepackage[mathcal]{eucal}
\usepackage{enumerate}
\usepackage[centertags]{amsmath}
\usepackage{graphics}
\usepackage[active]{srcltx}

\newtheorem{theorem}{Theorem}

\newtheorem{lemma}{Lemma}

\newtheorem*{rodin1}{Theorem R}
\newtheorem*{rodin2}{Theorem R2}
\newtheorem*{karagula}{Theorem K}
\newtheorem*{schipp}{Theorem Sch1}
\newtheorem*{ggk1}{Theorem GGK}

\newtheorem*{SCH2}{Theorem Sch2}

\begin{document}
\author{Ushangi Goginava}
\title[ Almost Everywhere Summability]{ Almost Everywhere Strong Summability
of Fej\'er means of rectangular partial sums of two-dimensional
Walsh-Fourier Series}
\address{U. Goginava, Department of Mathematics, Faculty of Exact and
Natural Sciences, Ivane Javakhishvili Tbilisi State University,
Chavcha\-vadze str. 1, Tbilisi 0128, Georgia}
\email{zazagoginava@gmail.com}
\date{}
\maketitle

\begin{abstract}
It is proved a BMO-estimation for rectangular partial sums of
two-dimensional Walsh-Fourier series from which it is derived an almost
everywhere exponential summability of rectangular partial sums of double
Walsh-Fourier series.
\end{abstract}

\footnotetext{%
2010 Mathematics Subject Classification 42C10 .
\par
The research was supported by Shota Rustaveli National Science Foundation
grant no.DI/9/5-100/13 (Function spaces, weighted inequalities for integral
operators and problems of summability of Fourier series)}

\section{\protect\bigskip\ Introduction}

We shall denote the set of all non-negative integers by $\mathbb{N}$ , the
set of all integers by$\,\,\mathbb{Z}$ and the set of dyadic rational
numbers in the unit interval $\mathbb{I}:=[0,1)$ by $\mathbb{Q}$. In
particular, each element of $\mathbb{Q}$ has the form $\frac{p}{2^{n}}$ for
some $p,n\in \mathbb{N},\,\,\,0\leq p\leq 2^{n}$.

Let $r_{0}\left( x\right) $ be the function defined by 
\begin{equation*}
r_{0}\left( x\right) =\left\{ 
\begin{array}{c}
1,\mbox{ if }x\in [0,1/2) \\ 
-1,\mbox{ if }x\in [1/2,1)%
\end{array}
\right. ,\,\,\,\,\,\,\,\,\,\,\,r_{0}\left( x+1\right) =r_{0}\left( x\right) .
\end{equation*}
The Rademacher system is defined by 
\begin{equation*}
r_{n}\left( x\right) =r_{0}\left( 2^{n}x\right) ,\,\,\,\,\,\,\,n\geq 1.
\end{equation*}

Let $w_{0},w_{1},...\,\,$represent the Walsh functions, i.e. $w_{0}\left(
x\right) =1\,\,$and $\,$if$\,\,k=2^{n_{1}}+\cdots +2^{n_{s}}\,$is a positive
integer with $n_{1}>n_{2}>\cdots >n_{s}\,\,$then 
\begin{equation*}
w_{k}\left( x\right) =r_{n_{1}}\left( x\right) \cdots r_{n_{s}}\left(
x\right) .
\end{equation*}

The Walsh-Dirichlet kernel is defined by 
\begin{equation*}
D_{0}\left( x\right) =0,D_{n}\left( x\right)
=\sum\limits_{k=0}^{n-1}w_{k}\left( x\right) \text{ }n\geq 1.
\end{equation*}

Given $x\in \mathbb{I}$, the expansion

\begin{equation}
x=\sum\limits_{k=0}^{\infty }x_{k}2^{-(k+1)},  \label{rep}
\end{equation}%
where each $x_{k}=0$ or $1$, will be called a dyadic expansion of $x.$ If $%
x\in \mathbb{I}\backslash \mathbb{Q}\mathbf{\,}$, then (\ref{rep}) is
uniquely determined. For the dyadic expansion $x\in \mathbb{Q}$\textbf{\ }we
choose the one for which $\lim\limits_{k\rightarrow \infty }x_{k}=0$.

The dyadic addition of $x,y\in \mathbb{I}$ in terms of the dyadic expansion
of $x$ and $y$ is defined by (see \cite{GES} or \cite{SWS}) 
\begin{equation*}
x\dotplus y=\sum\limits_{k=0}^{\infty }\left\vert x_{k}-y_{k}\right\vert
2^{-(k+1)}.
\end{equation*}

Denote $I_{N}:=[0,2^{-N}),$ $I_{N}\left( x\right) :=x\dotplus I_{N}$.

We consider the double system $\left\{ w_{n}(x)\times w_{m}(y):\,n,m\in 
\mathbb{N}\right\} $ on the unit square $\mathbb{I}^{2}=\left[ 0,1\right)
\times \left[ 0,1\right) .$The notiation $a\lesssim b$ in the whole paper
stands for $a\leq c\cdot b$, where $c$ is an absolute constant.

The norm (or quasinorm) of the space $L_{p}\left( \mathbb{I}^{2}\right) $ is
defined by 
\begin{equation*}
\left\Vert f\right\Vert _{p}:=\left( \int\limits_{\mathbb{I}^{2}}\left\vert
f\right\vert ^{p}\right) ^{1/p}\,\,\,\,\left( 0<p<+\infty \right) .
\end{equation*}

If $f\in L_{1}\left( \mathbb{I}^{2}\right) ,$ then 
\begin{equation*}
\hat{f}\left( n,m\right) =\int\limits_{\mathbb{I}^{2}}f\left(
x_{1},x_{2}\right) w_{n}(x_{1})w_{m}(x_{2})dx_{1}dx_{2}
\end{equation*}%
is the $\left( n,m\right) $-th Fourier coefficient of $f.$

The rectangular partial sums of double Fourier series with respect to the
Walsh system are defined by 
\begin{equation*}
S_{M,N}\left( x_{1},x_{2};f\right)
=\sum\limits_{m=0}^{M-1}\sum\limits_{n=0}^{N-1}\hat{f}\left( m,n\right)
w_{m}(x_{1})w_{n}(x_{2}).
\end{equation*}

Denote 
\begin{equation*}
S_{n}^{\left( 1\right) }\left( x_{1},x_{2};f\right) :=\sum\limits_{l=0}^{n-1}%
\widehat{f}\left( l,x_{2}\right) w_{l}\left( x_{1}\right) ,
\end{equation*}%
\begin{equation*}
S_{m}^{\left( 2\right) }\left( x_{1},x_{2};f\right) :=\sum\limits_{r=0}^{m-1}%
\widehat{f}\left( x_{1},r\right) w_{r}\left( x_{2}\right) ,
\end{equation*}%
where 
\begin{equation*}
\widehat{f}\left( l,x_{2}\right) =\int\limits_{\mathbb{I}}f\left(
x_{1},x_{2}\right) w_{l}\left( x_{1}\right) dx_{1},\widehat{f}\left(
x_{1},r\right) =\int\limits_{\mathbb{I}}f\left( x_{1},x_{2}\right)
w_{r}\left( x_{2}\right) dx_{2}.
\end{equation*}

Recall the definition of $BMO\left[ \mathbb{I}^{2}\right] $ space. Let $f\in
L_{1}\left( \mathbb{I}^{2}\right) $ . We say that $f$ has bounded mean
oscilation $\left( f\in BMO\left[ \mathbb{I}^{2}\right] \right) $ if%
\begin{equation*}
\left\Vert f\right\Vert _{BMO}:=\sup\limits_{Q}\left( \frac{1}{\left\vert
Q\right\vert }\int\limits_{Q}\left\vert f-f_{Q}\right\vert ^{2}\right)
^{1/2}<\infty ,
\end{equation*}%
where%
\begin{equation*}
f_{Q}:=\frac{1}{\left\vert Q\right\vert }\int\limits_{Q}f
\end{equation*}%
and the supremum is taken over all dyadic squares $Q\subset \mathbb{I}^{2}$.

Let $\xi :=\left\{ \xi _{n_{1}n_{2}}:n_{1},n_{2}=0,1,2,...\right\} $ be an
arbitrary sequence of numbers. Taking%
\begin{equation*}
\delta _{k}^{n}:=\left[ \frac{k}{2^{n}},\frac{k+1}{2^{n}}\right) ,
\end{equation*}%
we define%
\begin{equation*}
BMO\left[ \xi \right] :=\sup\limits_{0\leq n_{1},n_{2}<\infty }\left\Vert
\sum\limits_{k_{1}=0}^{2^{n_{1}}-1}\sum\limits_{k_{2}=0}^{2^{n_{2}}-1}\xi
_{k_{1}k_{2}}\mathbb{I}_{\delta _{k_{1}}^{n_{1}}}\left( t_{1}\right) \mathbb{%
I}_{\delta _{k_{2}}^{n_{2}}}\left( t_{2}\right) \right\Vert _{BMO},
\end{equation*}%
where $\mathbb{I}_{E}$ is the characteristic function of $E\subset \mathbb{I}%
^{2}$.

We denote by $L\left( \log L\right) ^{\alpha }\left( \mathbb{I}^{2}\right) $
the class of measurable functions $f$, with 
\begin{equation*}
\int\limits_{\mathbb{I}^{2}}|f|\left( \log ^{+}|f|\right) ^{\alpha }<\infty ,
\end{equation*}%
where $\log ^{+}u:=\mathbb{I}_{(1,\infty )}\log u.$

Denote by $S_{n}^{T}(x,f)$ the partial sums of the trigonometric Fourier
series of $f$ and let 
\begin{equation*}
\sigma _{n}^{T}(x,f)=\frac{1}{n+1}\sum_{k=0}^{n}S_{k}^{T}(x,f)
\end{equation*}%
be the $(C,1)$ means. Fejér \cite{Fe} proved that $\sigma _{n}^{T}(f)$
converges to $f$ uniformly for any $2\pi $-periodic continuous function.
Lebesgue in \cite{Le} established almost everywhere convergence of $(C,1)$
means if $f\in L_{1}(\mathbb{T}),\mathbb{T}:=[-\pi ,\pi )$. The strong
summability problem, i.e. the convergence of the strong means 
\begin{equation}
\frac{1}{n+1}\sum\limits_{k=0}^{n}\left\vert S_{k}^{T}\left( x,f\right)
-f\left( x\right) \right\vert ^{p},\quad x\in \mathbb{T},\quad p>0,
\label{Hp}
\end{equation}%
was first considered by Hardy and Littlewood in \cite{H-L}. They showed that
for any $f\in L_{r}(\mathbb{T})~\left( 1<r<\infty \right) $ the strong means
tend to $0$ a.e., if $n\rightarrow \infty $. The Fourier series of $f\in
L_{1}(\mathbb{T})$ is said to be $\left( H,p\right) $-summable at $x\in T$,
if the values (\ref{Hp}) converge to $0$ as $n\rightarrow \infty $. The $%
\left( H,p\right) $-summability problem in $L_{1}(\mathbb{T})$ has been
investigated by Marcinkiewicz \cite{Ma} for $p=2$, and later by Zygmund \cite%
{Zy2} for the general case $1\leq p<\infty $. Oskolkov in \cite{Os} proved
the following: Let $f\in L_{1}(\mathbb{T})$ and let $\Phi $ be a continuous
positive convex function on $[0,+\infty )$ with $\Phi \left( 0\right) =0$
and 
\begin{equation*}
\ln \Phi \left( t\right) =O\left( t/\ln \ln t\right) \text{ \ \ \ }\left(
t\rightarrow \infty \right) .
\end{equation*}%
Then for almost all $x$%
\begin{equation}
\lim\limits_{n\rightarrow \infty }\frac{1}{n+1}\sum\limits_{k=0}^{n}\Phi
\left( \left\vert S_{k}^{T}\left( x,f\right) -f\left( x\right) \right\vert
\right) =0.  \label{osk}
\end{equation}

It was noted in \cite{Os} that Totik announced the conjecture that (\ref{osk}%
) holds almost everywhere for any $f\in L_{1}(\mathbb{T})$, provided 
\begin{equation*}
\ln \Phi \left( t\right) =O\left( t\right) \quad \left( t\rightarrow \infty
\right) .
\end{equation*}%
In \cite{Ro} Rodin proved

\begin{rodin1}
Let $f\in L_{1}(\mathbb{T})$. Then for any $A>0$ 
\begin{equation*}
\lim\limits_{n\rightarrow \infty }\frac{1}{n+1}\sum\limits_{k=0}^{n}\left(
\exp \left( A\left\vert S_{k}^{T}\left( x,f\right) -f\left( x\right)
\right\vert \right) -1\right) =0
\end{equation*}%
for a. e. $x\in \mathbb{T}$.
\end{rodin1}

Karagulyan \cite{Ka} proved that the following is true.

\begin{karagula}
Suppose that a continuous increasing function $\Phi :[0,\infty )\rightarrow
\lbrack 0,\infty ),\Phi \left( 0\right) =0$, satisfies the condition%
\begin{equation*}
\limsup_{t\rightarrow +\infty }\frac{\log \Phi \left( t\right) }{t}=\infty .
\end{equation*}%
Then there exists a function $f\in L_{1}(\mathbb{T})$ for which the relation%
\begin{equation*}
\limsup_{n\rightarrow \infty }\frac{1}{n+1}\sum\limits_{k=0}^{n}\Phi \left(
\left\vert S_{k}^{T}\left( x,f\right) \right\vert \right) =\infty
\end{equation*}%
holds everywhere on $\mathbb{T}$.
\end{karagula}

For Walsh system Rodin \cite{rodin} (see also Schipp \cite{Sch}) proved that
the following is true.

\begin{rodin2}[Rodin]
If $\Phi (t):[0,\infty )\rightarrow \lbrack 0,\infty )$, $\Phi (0)=0$, is an
increasing continuous function satisfying 
\begin{equation}
\limsup_{t\rightarrow \infty }\frac{\log \Phi (t)}{t}<\infty ,  \label{rod}
\end{equation}%
then the partial sums of Walsh-Fourier series of any function $f\in
L_{1}\left( \mathbb{I}\right) $ satisfy the condition%
\begin{equation*}
\lim\limits_{n\rightarrow \infty }\frac{1}{n}\sum\limits_{k=1}^{n}\Phi
\left( \left\vert S_{k}\left( x;f\right) -f\left( x\right) \right\vert
\right) =0
\end{equation*}
\end{rodin2}

almost everywhere on $\mathbb{I}$.

In the paper \cite{JMAA} we established, that, as in trigonometric case \cite%
{Ka}, the bound (\ref{rod}) is sharp for a.e. $\Phi $-summability of
Walsh-Fourier series. Moreover, we prove

\begin{ggk1}
If an increasing function $\Phi (t):[0,\infty )\rightarrow \lbrack 0,\infty
) $ satisfies the condition%
\begin{equation*}
\limsup_{t\rightarrow \infty }\frac{\log \Phi (t)}{t}=\infty ,
\end{equation*}%
then there exists a function $f\in L_{1}\left( \mathbb{I}\right) $ such that 
\begin{equation*}
\limsup_{n\rightarrow \infty }\frac{1}{n}\sum\limits_{k=1}^{n}\Phi \left(
\left\vert S_{k}\left( x;f\right) \right\vert \right) =\infty
\end{equation*}%
holds everywhere on $\mathbb{I}$.
\end{ggk1}

The two-dimensional Fej\'er summability of $f\in L\log ^{+}L(\mathbb{I}^{2})$
was proved by Zygmund \cite{Zy2} for trigonometric Fourier series and by Mó%
ricz, Schipp and Wade \cite{msw} (see also Weisz \cite{WeC2}) for
Walsh-Fourier series. The two-dimensional strong summability, i. e.%
\begin{equation*}
\frac{1}{2^{n_{1}+n_{2}}}\sum\limits_{k_{1}=0}^{n_{1}-1}\sum%
\limits_{k_{2}=0}^{n_{2}-1}\left\vert S_{k_{1}k_{2}}\left(
x_{1},x_{2};f\right) -f\left( x_{1},x_{2}\right) \right\vert ^{p}\rightarrow
0\text{ \ \ a. e. \ as }n\rightarrow \infty
\end{equation*}%
was shown by Gogoladze \cite{GogStr} for trigonometric Fourier series and
for $f\in L\log ^{+}L(\mathbb{T}^{2})$. The same result for more-dimensional
Walsh-Fourier series is due to Rodin \cite{Ro1} (see also Weisz \cite{We}).
These results show that in the case of two dimensional functions the $%
(C;1,1) $ summability and $(C;1,1)$ strong summability we have the same
maximal convergence spaces. That is, in both cases we have $L\log ^{+}L$.

It is proved in (\cite{CA2}) a BMO-estimation for quadratic partial sums of
two-dimensional trigonometric Fourier series from which it is derived an
almost everywhere exponential summability of quadratic partial sums of
double Fourier series.

The results on strong summation and approximation of trigonometric Fourier
series have been extended for several other orthogonal systems. For
instance, concerning the Walsh system see Schipp \cite{Sch1,Sch2,Sch3},
Fridli and Schipp \cite{FS,FS2}, Leindler \cite{Le1,Le2,Le3,Le4}, Totik \cite%
{To1,To2,To3}, Rodin \cite{Ro1}, Weisz \cite{We,WeL}, Gabisonia \cite{Ga},
Goginava, Gogoladze \cite{GGArm}.

The problems of summability of multiple Fourier series have been
investigated by Gogoladze \cite{Gog1,Gog2}, Wang \cite{Wa}, Zhag \cite{ZhHe}%
, Glukhov \cite{Gl}, Goginava \cite{GoJAT, GoSM}, Goginava, Gogoladze \cite%
{GG}, Gat, Goginava, Karagulyan \cite{AM}.

\section{Main Results}

In this paper we study a BMO-estimation for rectangular partial sums of
two-dimensional Walsh-Fourier series.

\begin{theorem}
\label{bmo}If $f\in L\left( \log L\right) ^{2}\left( \mathbb{I}^{2}\right) $%
, then%
\begin{equation*}
\left\vert \left\{ \left( x_{1},x_{2}\right) \in \mathbb{I}^{2}:\text{BMO}%
\left[ S_{n_{1}n_{2}}\left( x_{1},x_{2};f\right) \right] >\lambda \right\}
\right\vert \lesssim \frac{1}{\lambda }\left( 1+\int\limits_{\mathbb{I}%
^{2}}|f|\left( \log |f|\right) ^{2}\right) .
\end{equation*}
\end{theorem}

The following theorem shows that the rectangular sums of two-dimensional
Walsh-Fourier series of a function $f\in L\left( \log L\right) ^{2}\left( 
\mathbb{I}^{2}\right) $ are almost everywhere exponentially summable to the
function $f$. It will be obtained from the previous theorem (see \cite{CA2})
by using the John-Nirenberg theorem.

\begin{theorem}
\label{a.e.exp} Suppose that $f\in L\left( \log L\right) ^{2}\left( \mathbb{I%
}^{2}\right) $. Then for any $A>0$%
\begin{equation*}
\lim\limits_{m_{1},m_{2}\rightarrow \infty }\frac{1}{2^{m_{1}+m_{2}}}%
\sum\limits_{n_{1}=0}^{2^{m_{1}}-1}\sum\limits_{n_{2}=0}^{2^{m_{2}}-1}\left(
\exp \left( A\left\vert S_{n_{1}n_{2}}\left( x_{1},x_{2};f\right) -f\left(
x_{1},x_{2}\right) \right\vert \right) -1\right) =0
\end{equation*}%
for a. e. $\left( x_{1},x_{2}\right) \in \mathbb{I}^{2}$.
\end{theorem}

\section{Auxiliary Results}

Schipp in \cite{Sch} introduce the following operator%
\begin{equation*}
V_{n}\left( x;f\right) :=\left( \sum\limits_{l=0}^{2^{n}-1}\left(
\int\limits_{l2^{-n}}^{\left( l+1\right)
2^{-n}}\sum\limits_{j=0}^{n-1}2^{j-1}\mathbb{I}_{I_{j}}\left( t\right)
S_{2^{n}}\left( x\dotplus t\dotplus e_{j},f\right) dt\right) ^{2}\right)
^{1/2},
\end{equation*}

Let%
\begin{equation*}
V\left( f\right) :=\sup\limits_{n}V_{n}\left( f\right) .
\end{equation*}

The following theorem is proved by Schipp.

\begin{schipp}[\protect\cite{Sch}]
\label{SchD} Let $f\in L_{1}\left( \mathbb{I}\right) $. Then%
\begin{equation*}
\mu \left\{ \left\vert Vf\right\vert >\lambda \right\} \lesssim \frac{%
\left\Vert f\right\Vert _{1}}{\lambda }
\end{equation*}%
and%
\begin{equation*}
\left\Vert Vf\right\Vert \lesssim 1+\int\limits_{\mathbb{I}}\left\vert
f\right\vert \log ^{+}\left\vert f\right\vert \text{ \ \ }\left( f\in L\log
^{+}L\left( \mathbb{I}\right) \right) .
\end{equation*}
\end{schipp}

\begin{SCH2}[\protect\cite{Sch}]
\label{SchD2}The following estimation holds%
\begin{equation*}
\left\{ \frac{1}{2^{n}}\sum\limits_{m=0}^{2^{n}-1}\left\vert S_{m}\left(
x;f\right) \right\vert ^{2}\right\} ^{1/2}\lesssim V_{n}\left( x;\left\vert
f\right\vert \right) .
\end{equation*}
\end{SCH2}

Set%
\begin{equation*}
V_{m_{1}m_{2}}\left( x_{1},x_{2};f\right)
\end{equation*}%
\begin{equation*}
:=\left(
\sum\limits_{l_{1}=0}^{2^{m_{1}}-1}\sum\limits_{l_{2}=0}^{2^{m_{2}}-1}\left(
\int\limits_{l_{1}2^{-m_{1}}}^{\left( l_{1}+1\right)
2^{-m_{1}}}\int\limits_{l_{2}2^{-m_{2}}}^{\left( l_{2}+1\right)
2^{-m_{2}}}\sum\limits_{j_{1}=0}^{m_{1}-1}2^{j_{1}-1}\mathbb{I}%
_{I_{j_{1}}}\left( t_{1}\right) \sum\limits_{j_{2}=0}^{m_{2}-1}2^{j_{2}-1}%
\mathbb{I}_{I_{j_{2}}}\left( t_{2}\right) \right. \right.
\end{equation*}%
\begin{equation*}
\left. \left. \times S_{2^{m_{1}}2^{m_{2}}}\left( x_{1}\dotplus
t_{1}\dotplus e_{j_{1}},x_{2}\dotplus t_{2}\dotplus e_{j_{2}};f\right)
dt_{1}dt_{2}\right) ^{2}\right) ^{1/2}.
\end{equation*}

For a two-dimensional integrable function $f$ we need to introduce the
following functions 
\begin{eqnarray}
&&V_{n}^{\left( 1\right) }\left( x_{1},x_{2};f\right)  \label{V1} \\
&:&=\left( \sum\limits_{l=0}^{2^{n}-1}\left( \int\limits_{l2^{-n}}^{\left(
l+1\right) 2^{-n}}\sum\limits_{j=0}^{n-1}2^{j-1}\mathbb{I}_{I_{j}}\left(
t\right) S_{2^{n}}^{(1)}\left( x_{1}\dotplus t\dotplus e_{j},x_{2};f\right)
^{2}dt\right) ^{2}\right) ^{1/2},  \notag
\end{eqnarray}%
\begin{eqnarray}
&&V_{n}^{\left( 2\right) }\left( x_{1},x_{2};f\right)  \label{V2} \\
&:&=\left( \sum\limits_{l=0}^{2^{n}-1}\left( \int\limits_{l2^{-n}}^{\left(
l+1\right) 2^{-n}}\sum\limits_{j=0}^{n-1}2^{j-1}\mathbb{I}_{I_{j}}\left(
t\right) S_{2^{n}}^{(2)}\left( x_{1},x_{2}\dotplus t\dotplus e_{j};f\right)
dt\right) ^{2}\right) ^{1/2},  \notag
\end{eqnarray}%
\begin{equation*}
V^{\left( s\right) }\left( x_{1},x_{2};f\right) :=\sup\limits_{n}\left\vert
V_{n}^{\left( s\right) }\left( x_{1},x_{2};f\right) \right\vert ,\text{ \ \ }%
s=1,2.
\end{equation*}

\begin{lemma}
\label{2-estimation}The following estimation holds%
\begin{equation*}
\left\{ \frac{1}{2^{m_{1}+m_{2}}}\sum\limits_{n_{1}=0}^{2^{m_{1}}-1}\sum%
\limits_{n_{2}=0}^{2^{m_{2}}-1}\left\vert f\ast \left( D_{n_{1}}\otimes
D_{n_{2}}\right) \right\vert ^{2}\right\} ^{1/2}\lesssim
V_{m_{1}m_{2}}\left( x_{1},x_{2};\left\vert f\right\vert \right) .
\end{equation*}
\end{lemma}

\begin{proof}[Proof of Lemma \protect\ref{2-estimation}]
Let%
\begin{equation*}
\varepsilon _{ji}:=\left\{ 
\begin{array}{l}
-1,\text{ if }j=0,1,...,i-1 \\ 
1,\text{ if }j=i%
\end{array}%
\right. .
\end{equation*}%
In \cite{Sch}, Schipp proved that%
\begin{eqnarray*}
D_{m}\left( t\right) &=&\sum\limits_{k=0}^{n-1}\mathbb{I}_{I_{k}\backslash
I_{k+1}}\left( t\right) \sum\limits_{j=0}^{k}\varepsilon
_{kj}2^{j-1}w_{m}\left( t+e_{j}\right) \\
&&-\frac{1}{2}w_{m}\left( t\right) +\left( m+1/2\right) \mathbb{I}%
_{I_{n}}\left( t\right) ,\text{ \ \ }m<2^{n}.
\end{eqnarray*}

Then we can write%
\begin{equation}
H_{m_{1}m_{2}}\left( x_{1},x_{2};f\right)  \label{R1-R9}
\end{equation}%
\begin{equation*}
:=\left\{ \frac{1}{2^{m_{1}+m_{2}}}\sum\limits_{n_{1}=0}^{2^{m_{1}}-1}\sum%
\limits_{n_{2}=0}^{2^{m_{2}}-1}\left\vert S_{2^{m_{1}}2^{m_{2}}}\left(
f\right) \ast \left( D_{n_{1}}\otimes D_{n_{2}}\right) \right\vert
^{2}\right\} ^{1/2}
\end{equation*}%
\begin{equation*}
\leq \left\{ \frac{1}{2^{m_{1}+m_{2}}}\sum\limits_{n_{1}=0}^{2^{m_{1}}-1}%
\sum\limits_{n_{2}=0}^{2^{m_{2}}-1}\left\vert \int\limits_{\mathbb{I}%
^{2}}S_{2^{m_{1}}2^{m_{2}}}\left( x_{1}\dotplus t_{1},x_{2}\dotplus
t_{2};f\right) \sum\limits_{k_{1}=0}^{m_{1}-1}\mathbb{I}_{I_{k_{1}}%
\backslash I_{k_{1}+1}}\left( t_{1}\right) \right. \right.
\end{equation*}%
\begin{equation*}
\times \sum\limits_{j_{1}=0}^{k_{1}}\varepsilon
_{k_{1}j_{1}}2^{j_{1}-1}w_{n_{1}}\left( t_{1}\dotplus e_{j_{1}}\right)
\sum\limits_{k_{2}=0}^{m_{2}-1}\mathbb{I}_{I_{k_{2}}\backslash
I_{k_{2}+1}}\left( t_{2}\right)
\end{equation*}%
\begin{equation*}
\left. \left. \times \sum\limits_{j_{2}=0}^{k_{2}}\varepsilon
_{k_{2}j_{2}}2^{j_{2}-1}w_{n_{2}}\left( t_{2}\dotplus e_{j_{2}}\right)
dt_{1}dt_{2}\right\vert ^{2}\right\} ^{1/2}
\end{equation*}%
\begin{equation*}
+\left\{ \frac{1}{2^{m_{1}+m_{2}}}\sum\limits_{n_{1}=0}^{2^{m_{1}}-1}\sum%
\limits_{n_{2}=0}^{2^{m_{2}}-1}\left\vert \int\limits_{\mathbb{I}%
^{2}}S_{2^{m_{1}}2^{m_{2}}}\left( x_{1}\dotplus t_{1},x_{2}\dotplus
t_{2};f\right) \sum\limits_{k_{1}=0}^{m_{1}-1}\mathbb{I}_{I_{k_{1}}%
\backslash I_{k_{1}+1}}\left( t_{1}\right) \right. \right.
\end{equation*}%
\begin{equation*}
\left. \left. \times \sum\limits_{j_{1}=0}^{k_{1}}\varepsilon
_{k_{1}j_{1}}2^{j_{1}-1}w_{n_{1}}\left( t_{1}\dotplus e_{j_{1}}\right) \frac{%
w_{n_{2}}\left( t_{2}\right) }{2}dt_{1}dt_{2}\right\vert ^{2}\right\} ^{1/2}
\end{equation*}%
\begin{equation*}
+\left\{ \frac{1}{2^{m_{1}+m_{2}}}\sum\limits_{n_{1}=0}^{2^{m_{1}}-1}\sum%
\limits_{n_{2}=0}^{2^{m_{2}}-1}\left\vert \int\limits_{\mathbb{I}%
^{2}}S_{2^{m_{1}}2^{m_{2}}}\left( x_{1}\dotplus t_{1},x_{2}\dotplus
t_{2};f\right) \sum\limits_{k_{1}=0}^{m_{1}-1}\mathbb{I}_{I_{k_{1}}%
\backslash I_{k_{1}+1}}\left( t_{1}\right) \right. \right.
\end{equation*}%
\begin{equation*}
\left. \left. \times \sum\limits_{j_{1}=0}^{k_{1}}\varepsilon
_{k_{1}j_{1}}2^{j_{1}-1}w_{n_{1}}\left( t_{1}\dotplus e_{j_{1}}\right)
\left( n_{2}+1/2\right) \mathbb{I}_{I_{m_{2}}}\left( t_{2}\right)
dt_{1}dt_{2}\right\vert ^{2}\right\} ^{1/2}
\end{equation*}%
\begin{equation*}
+\left\{ \frac{1}{2^{m_{1}+m_{2}}}\sum\limits_{n_{1}=0}^{2^{m_{1}}-1}\sum%
\limits_{n_{2}=0}^{2^{m_{2}}-1}\left\vert \int\limits_{\mathbb{I}%
^{2}}S_{2^{m_{1}}2^{m_{2}}}\left( x_{1}\dotplus t_{1},x_{2}\dotplus
t_{2};f\right) \sum\limits_{k_{2}=0}^{m_{2}-1}\mathbb{I}_{I_{k_{2}}%
\backslash I_{k_{2}+1}}\left( t_{2}\right) \right. \right.
\end{equation*}%
\begin{equation*}
\left. \left. \times \sum\limits_{j_{2}=0}^{k_{2}}\varepsilon
_{k_{2}j_{2}}2^{j_{2}-1}w_{n_{2}}\left( t_{2}\dotplus e_{j_{2}}\right) \frac{%
w_{n_{1}}\left( t_{1}\right) }{2}dt_{1}dt_{2}\right\vert ^{2}\right\} ^{1/2}
\end{equation*}%
\begin{equation*}
+\left\{ \frac{1}{2^{m_{1}+m_{2}}}\sum\limits_{n_{1}=0}^{2^{m_{1}}-1}\sum%
\limits_{n_{2}=0}^{2^{m_{2}}-1}\int\limits_{\mathbb{I}%
^{2}}S_{2^{m_{1}}2^{m_{2}}}\left( x_{1}\dotplus t_{1},x_{2}\dotplus
t_{2};f\right) \right.
\end{equation*}%
\begin{equation*}
\left. \left. \times \frac{w_{n_{1}}\left( t_{1}\right) }{2}\frac{%
w_{n_{2}}\left( t_{2}\right) }{2}dt_{1}dt_{2}\right\vert ^{2}\right\} ^{1/2}
\end{equation*}%
\begin{equation*}
+\left\{ \frac{1}{2^{m_{1}+m_{2}}}\sum\limits_{n_{1}=0}^{2^{m_{1}}-1}\sum%
\limits_{n_{2}=0}^{2^{m_{2}}-1}\left\vert \int\limits_{\mathbb{I}%
^{2}}S_{2^{m_{1}}2^{m_{2}}}\left( x_{1}\dotplus t_{1},x_{2}\dotplus
t_{2};f\right) \right. \right.
\end{equation*}%
\begin{equation*}
\left. \left. \times \frac{w_{n_{1}}\left( t_{1}\right) }{2}\left(
n_{2}+1/2\right) \mathbb{I}_{I_{m_{2}}}\left( t_{2}\right)
dt_{1}dt_{2}\right\vert ^{2}\right\} ^{1/2}
\end{equation*}%
\begin{equation*}
+\left\{ \frac{1}{2^{m_{1}+m_{2}}}\sum\limits_{n_{1}=0}^{2^{m_{1}}-1}\sum%
\limits_{n_{2}=0}^{2^{m_{2}}-1}\left\vert \int\limits_{\mathbb{I}%
^{2}}S_{2^{m_{1}}2^{m_{2}}}\left( x_{1}\dotplus t_{1},x_{2}\dotplus
t_{2};f\right) \sum\limits_{k_{2}=0}^{m_{2}-1}\mathbb{I}_{I_{k_{2}}%
\backslash I_{k_{2}+1}}\left( t_{2}\right) \right. \right.
\end{equation*}%
\begin{equation*}
\left. \left. \times \sum\limits_{j_{2}=0}^{k_{2}}\varepsilon
_{k_{2}j_{2}}2^{j_{2}-1}w_{n_{2}}\left( t_{2}\dotplus e_{j_{2}}\right)
\left( n_{1}+1/2\right) \mathbb{I}_{I_{m_{1}}}\left( t_{1}\right)
dt_{1}dt_{2}\right\vert ^{2}\right\} ^{1/2}
\end{equation*}%
\begin{equation*}
+\left\{ \frac{1}{2^{m_{1}+m_{2}}}\sum\limits_{n_{1}=0}^{2^{m_{1}}-1}\sum%
\limits_{n_{2}=0}^{2^{m_{2}}-1}\left\vert \int\limits_{\mathbb{I}%
^{2}}S_{2^{m_{1}}2^{m_{2}}}\left( x_{1}+t_{1},x_{2}+t_{2};f\right) \right.
\right.
\end{equation*}%
\begin{equation*}
\left. \left. \times \frac{w_{n_{2}}\left( t_{2}\right) }{2}\left(
n_{1}+1/2\right) \mathbb{I}_{I_{m_{1}}}\left( t_{1}\right)
dt_{1}dt_{2}\right\vert ^{2}\right\} ^{1/2}
\end{equation*}%
\begin{equation*}
+\left\{ \frac{1}{2^{m_{1}+m_{2}}}\sum\limits_{n_{1}=0}^{2^{m_{1}}-1}\sum%
\limits_{n_{2}=0}^{2^{m_{2}}-1}\left\vert \int\limits_{\mathbb{I}%
^{2}}S_{2^{m_{1}}2^{m_{2}}}\left( x_{1}\dotplus t_{1},x_{2}\dotplus
t_{2};f\right) \right. \right.
\end{equation*}%
\begin{equation*}
\left. \left. \times \left( n_{1}+1/2\right) \mathbb{I}_{I_{m_{1}}}\left(
t_{1}\right) \left( n_{2}+1/2\right) \mathbb{I}_{I_{m_{2}}}\left(
t_{2}\right) dt_{1}dt_{2}\right\vert ^{2}\right\} ^{1/2}
\end{equation*}%
\begin{equation*}
:=\sum\limits_{i=1}^{9}R_{i}.
\end{equation*}

There is a suitable vector%
\begin{equation*}
\left\{ \beta _{n_{1}n_{2}}^{\left( 1\right) }\left( x_{1},x_{2}\right)
:0\leq n_{1}<2^{m_{1}},0\leq n_{2}<2^{m_{2}}\right\}
\end{equation*}%
such that%
\begin{equation*}
\sum\limits_{n_{1}=0}^{2^{m_{1}}-1}\sum\limits_{n_{2}=0}^{2^{m_{2}}-1}\left%
\vert \beta _{n_{1}n_{2}}^{\left( 1\right) }\left( x_{1},x_{2}\right)
\right\vert ^{2}=1
\end{equation*}%
and%
\begin{equation}
2^{\left( m_{1}+m_{2}\right) /2}R_{1}  \label{R1}
\end{equation}%
\begin{equation*}
=\int\limits_{\mathbb{I}^{2}}\sum\limits_{j_{1}=0}^{m_{1}-1}2^{j_{1}-1}%
\left( \sum\limits_{k_{1}=j_{1}}^{m_{1}-1}\varepsilon _{k_{1}j_{1}}\mathbb{I}%
_{I_{k_{1}}\backslash I_{k_{1}+1}}\left( t_{1}\dotplus e_{j_{1}}\right)
\right)
\end{equation*}%
\begin{equation*}
\times \sum\limits_{j_{2}=0}^{m_{2}-1}2^{j_{1}-1}\left(
\sum\limits_{k_{2}=j_{2}}^{m_{2}-1}\varepsilon _{k_{2}j_{2}}\mathbb{I}%
_{I_{k_{2}}\backslash I_{k_{2}+1}}\left( t_{1}\dotplus e_{j_{2}}\right)
\right)
\end{equation*}%
\begin{equation*}
\times S_{2^{m_{1}}2^{m_{2}}}\left( x_{1}\dotplus t_{1}\dotplus
e_{j_{1}},x_{2}\dotplus t_{2}\dotplus e_{j_{2}};f\right)
\end{equation*}%
\begin{equation*}
\times
\sum\limits_{n_{1}=0}^{2^{m_{1}}-1}\sum\limits_{n_{2}=0}^{2^{m_{2}}-1}\beta
_{n_{1}n_{2}}^{\left( 1\right) }\left( x_{1},x_{2}\right) w_{n_{1}}\left(
t_{1}\right) w_{n_{2}}\left( t_{2}\right) dt_{1}dt_{2}
\end{equation*}%
\begin{equation*}
\leq \int\limits_{\mathbb{I}^{2}}\sum\limits_{j_{1}=0}^{m_{1}-1}2^{j_{1}-1}%
\mathbb{I}_{I_{j_{1}}}\left( t_{1}\right)
\sum\limits_{j_{2}=0}^{m_{2}-1}2^{j_{2}-1}\mathbb{I}_{I_{j_{2}}}\left(
t_{2}\right)
\end{equation*}%
\begin{equation*}
\times S_{2^{m_{1}}2^{m_{2}}}\left(
x_{1}+t_{1}+e_{j_{1}},x_{2}+t_{2}+e_{j_{2}};\left\vert f\right\vert \right)
\end{equation*}%
\begin{equation*}
\times \left\vert
\sum\limits_{n_{1}=0}^{2^{m_{1}}-1}\sum\limits_{n_{2}=0}^{2^{m_{2}}-1}\beta
_{n_{1}n_{2}}^{\left( 1\right) }\left( x_{1},x_{2}\right) w_{n_{1}}\left(
t_{1}\right) w_{n_{2}}\left( t_{2}\right) \right\vert dt_{1}dt_{2}.
\end{equation*}

Analogously, we can prove that 
\begin{equation}
2^{\left( m_{1}+m_{2}\right) /2}R_{2}\lesssim \int\limits_{\mathbb{I}%
^{2}}\sum\limits_{j_{1}=0}^{m_{1}-1}2^{j_{1}-1}\mathbb{I}_{I_{j_{1}}}\left(
t_{1}\right)  \label{R2}
\end{equation}%
\begin{equation*}
\times S_{2^{m_{1}}2^{m_{2}}}\left( x_{1}\dotplus t_{1}\dotplus
e_{j_{1}},x_{2}\dotplus t_{2}\dotplus e_{0};\left\vert f\right\vert \right)
\end{equation*}%
\begin{equation*}
\times \left\vert
\sum\limits_{n_{1}=0}^{2^{m_{1}}-1}\sum\limits_{n_{2}=0}^{2^{m_{2}}-1}\beta
_{n_{1}n_{2}}^{\left( 2\right) }\left( x_{1},x_{2}\right) w_{n_{1}}\left(
t_{1}\right) w_{n_{2}}\left( e_{0}\right) w_{n_{2}}\left( t_{2}\right)
\right\vert dt_{1}dt_{2},
\end{equation*}

\begin{equation}
2^{\left( m_{1}+m_{2}\right) /2}R_{4}  \label{R4}
\end{equation}%
\begin{equation*}
\lesssim \int\limits_{\mathbb{I}^{2}}\sum%
\limits_{j_{2}=0}^{m_{2}-1}2^{j_{2}-1}\mathbb{I}_{I_{j_{2}}}\left(
t_{2}\right) S_{2^{m_{1}}2^{m_{2}}}\left(
x_{1}+t_{1}+e_{0},x_{2}+t_{2}+e_{j_{2}};\left\vert f\right\vert \right)
\end{equation*}%
\begin{equation*}
\times \left\vert
\sum\limits_{n_{1}=0}^{2^{m_{1}}-1}\sum\limits_{n_{2}=0}^{2^{m_{2}}-1}\beta
_{n_{1}n_{2}}^{\left( 4\right) }\left( x_{1},x_{2}\right) w_{n_{1}}\left(
t_{1}\right) w_{n_{1}}\left( e_{0}\right) w_{n_{2}}\left( t_{2}\right)
\right\vert dt_{1}dt_{2}.
\end{equation*}%
\begin{equation}
2^{\left( m_{1}+m_{2}\right) /2}R_{5}  \label{R5}
\end{equation}%
\begin{equation*}
\lesssim \int\limits_{\mathbb{I}^{2}}S_{2^{m_{1}}2^{m_{2}}}\left(
x_{1}\dotplus t_{1}\dotplus e_{0},x_{2}\dotplus t_{2}\dotplus
e_{0};\left\vert f\right\vert \right)
\end{equation*}%
\begin{equation*}
\times \left\vert
\sum\limits_{n_{1}=0}^{2^{m_{1}}-1}\sum\limits_{n_{2}=0}^{2^{m_{2}}-1}\beta
_{n_{1}n_{2}}^{\left( 5\right) }\left( x_{1},x_{2}\right) w_{n_{1}}\left(
t_{1}\right) w_{n_{1}}\left( e_{0}\right) w_{n_{2}}\left( t_{2}\right)
w_{n_{2}}\left( e_{0}\right) \right\vert dt_{1}dt_{2},
\end{equation*}%
where%
\begin{equation*}
\sum\limits_{n_{1}=0}^{2^{m_{1}}-1}\left\vert \beta _{n_{1}}^{\left(
s\right) }\left( x_{1},x_{2}\right) \right\vert ^{2}=1\text{ }\left(
x_{1},x_{2}\right) \in \mathbb{I}^{2},s=2,4,5.
\end{equation*}

Now, we estimate $R_{3}$. There is a suitable vector%
\begin{equation*}
\left\{ \beta _{n_{1}}^{\left( 3\right) }\left( x_{1},x_{2}\right) :0\leq
n_{1}<2^{m_{1}}\right\}
\end{equation*}%
such that%
\begin{equation*}
\sum\limits_{n_{1}=0}^{2^{m_{1}}-1}\left\vert \beta _{n_{1}}^{\left(
3\right) }\left( x_{1},x_{2}\right) \right\vert ^{2}=1,\text{ }\left(
x_{1},x_{2}\right) \in \mathbb{I}^{2}
\end{equation*}%
and%
\begin{equation}
2^{\left( m_{1}+m_{2}\right) /2}R_{3}  \label{R3}
\end{equation}%
\begin{equation*}
\leq c2^{3m_{2}/2}\int\limits_{I\times
I_{m_{2}}}\sum\limits_{j_{1}=0}^{m_{1}-1}2^{j_{1}-1}\mathbb{I}%
_{I_{j_{1}}}\left( t_{1}\right) S_{2^{m_{1}}2^{m_{2}}}\left( x_{1}\dotplus
t_{1}\dotplus e_{j_{1}},x_{2}\dotplus t_{2};\left\vert f\right\vert \right)
\end{equation*}%
\begin{equation*}
\times \left\vert \sum\limits_{n_{1}=0}^{2^{m_{1}}-1}\beta _{n_{1}}^{\left(
3\right) }\left( x_{1},x_{2}\right) w_{n_{1}}\left( t_{1}\right) \right\vert
dt_{1}dt_{2}.
\end{equation*}

Analogously, we can prove that%
\begin{equation}
2^{\left( m_{1}+m_{2}\right) /2}R_{6}\lesssim 2^{\left( 3/2\right)
m_{2}}\int\limits_{\mathbb{I}\times I_{m_{2}}}S_{2^{m_{1}}2^{m_{2}}}\left(
x_{1}+t_{1}+e_{0},x_{2}+t_{2};\left\vert f\right\vert \right)  \label{R6}
\end{equation}%
\begin{equation*}
\times \left\vert \sum\limits_{n_{1}=0}^{2^{m_{1}}-1}\beta _{n_{1}}^{\left(
6\right) }\left( x_{1},x_{2}\right) w_{n_{1}}\left( e_{0}\right)
w_{n_{1}}\left( t_{1}\right) \right\vert dt_{1}dt_{2},
\end{equation*}%
\begin{equation}
2^{\left( m_{1}+m_{2}\right) /2}R_{7}\lesssim 2^{\left( 3/2\right)
m_{1}}\int\limits_{I_{m_{1}}\times \mathbb{I}}\sum%
\limits_{j_{2}=0}^{m_{2}-1}2^{j_{2}-1}\mathbb{I}_{I_{j_{2}}}\left(
t_{2}\right)  \label{R7}
\end{equation}%
\begin{equation*}
\times S_{2^{m_{1}}2^{m_{2}}}\left(
x_{1}+t_{1},x_{2}+t_{2}+e_{j_{2}};\left\vert f\right\vert \right)
\end{equation*}%
\begin{equation*}
\times \left\vert \sum\limits_{n_{1}=0}^{2^{m_{1}}-1}\beta _{n_{1}}^{\left(
7\right) }\left( x_{1},x_{2}\right) w_{n_{2}}\left( t_{2}\right) \right\vert
dt_{1}dt_{2},
\end{equation*}%
\begin{equation}
2^{\left( m_{1}+m_{2}\right) /2}R_{8}\lesssim 2^{\left( 3/2\right)
m_{1}}\int\limits_{I_{m_{1}}\times \mathbb{I}}S_{2^{m_{1}}2^{m_{2}}}\left(
x_{1}\dotplus t_{1},x_{2}\dotplus t_{2}\dotplus e_{0};\left\vert
f\right\vert \right)  \label{R8}
\end{equation}%
\begin{equation*}
\times \left\vert \sum\limits_{n_{2}=0}^{2^{m_{2}}-1}\beta _{n_{2}}^{\left(
8\right) }\left( x_{1},x_{2}\right) w_{n_{2}}\left( e_{0}\right)
w_{n_{2}}\left( t_{2}\right) \right\vert dt_{1}dt_{2},
\end{equation*}%
\begin{equation}
2^{\left( m_{1}+m_{2}\right) /2}R_{9}\lesssim 2^{\left( 3/2\right) \left(
m_{1}+m_{2}\right) }\int\limits_{I_{m_{1}}\times
I_{m_{2}}}S_{2^{m_{1}}2^{m_{2}}}\left( x_{1}\dotplus t_{1},x_{2}\dotplus
t_{2};\left\vert f\right\vert \right) ,  \label{R9}
\end{equation}%
where%
\begin{equation*}
\sum\limits_{n_{1}=0}^{2^{m_{1}}-1}\left\vert \beta _{n_{1}}^{\left(
s\right) }\left( x_{1},x_{2}\right) \right\vert ^{2}=1\text{ }\left(
x_{1},x_{2}\right) \in \mathbb{I}^{2},s=6,7
\end{equation*}%
and%
\begin{equation*}
\sum\limits_{n_{2}=0}^{2^{m_{2}}-1}\left\vert \beta _{n_{2}}^{\left(
8\right) }\left( x_{1},x_{2}\right) \right\vert ^{2}=1\text{ }\left(
x_{1},x_{2}\right) \in \mathbb{I}^{2}.
\end{equation*}

Set%
\begin{equation*}
P_{m_{1}m_{2}}^{\left( 1\right) }\left( x_{1},x_{2}\right)
:=\sum\limits_{n_{1}=0}^{2^{m_{1}}-1}\sum\limits_{n_{2}=0}^{2^{m_{2}}-1}%
\beta _{n_{1}n_{2}}^{\left( 1\right) }\left( x_{1},x_{2}\right)
w_{n_{1}}\left( t_{1}\right) w_{n_{2}}\left( t_{2}\right) .
\end{equation*}%
Then from (\ref{R1}) we have%
\begin{equation*}
2^{\left( m_{1}+m_{2}\right) /2}R_{1}
\end{equation*}%
\begin{equation*}
=\sum\limits_{l_{1}=0}^{2^{m_{1}}-1}\sum\limits_{l_{2}=0}^{2^{m_{2}}-1}\left%
\vert P_{m_{1}m_{2}}^{\left( 1\right) }\left( \frac{l_{1}}{2^{m_{1}}},\frac{%
l_{2}}{2^{m_{2}}}\right) \right\vert \int\limits_{l_{1}2^{-m_{1}}}^{\left(
l_{1}+1\right) 2^{-m_{1}}}\int\limits_{l_{2}2^{-m_{2}}}^{\left(
l_{2}+1\right) 2^{-m_{2}}}\sum\limits_{j_{1}=0}^{m_{1}-1}2^{j_{1}-1}\mathbb{I%
}_{I_{j_{1}}}\left( t_{1}\right)
\end{equation*}%
\begin{equation*}
\times \sum\limits_{j_{2}=0}^{m_{2}-1}2^{j_{2}-1}\mathbb{I}%
_{I_{j_{2}}}\left( t_{2}\right) S_{2^{m_{1}}2^{m_{2}}}\left( x_{1}\dotplus
t_{1}\dotplus e_{j_{1}},x_{2}\dotplus t_{2}\dotplus e_{j_{2}};\left\vert
f\right\vert \right) dt_{1}dt_{2}
\end{equation*}

\begin{equation*}
\leq \left(
\sum\limits_{l_{1}=0}^{2^{m_{1}}-1}\sum\limits_{l_{2}=0}^{2^{m_{2}}-1}\left%
\vert P_{m_{1}m_{2}}^{\left( 1\right) }\left( \frac{l_{1}}{2^{m_{1}}},\frac{%
l_{2}}{2^{m_{2}}}\right) \right\vert ^{2}\right) ^{1/2}
\end{equation*}%
\begin{equation*}
\times \left(
\sum\limits_{l_{1}=0}^{2^{m_{1}}-1}\sum\limits_{l_{2}=0}^{2^{m_{2}}-1}\left(
\int\limits_{l_{1}2^{-m_{1}}}^{\left( l_{1}+1\right)
2^{-m_{1}}}\int\limits_{l_{2}2^{-m_{2}}}^{\left( l_{2}+1\right)
2^{-m_{2}}}\sum\limits_{j_{1}=0}^{m_{1}-1}2^{j_{1}-1}\mathbb{I}%
_{I_{j_{1}}}\left( t_{1}\right) \sum\limits_{j_{2}=0}^{m_{2}-1}2^{j_{2}-1}%
\mathbb{I}_{I_{j_{2}}}\left( t_{2}\right) \right. \right.
\end{equation*}%
\begin{equation*}
\left. \left. \times S_{2^{m_{1}}2^{m_{2}}}\left( x_{1}\dotplus
t_{1}\dotplus e_{j_{1}},x_{2}\dotplus t_{2}\dotplus e_{j_{2}};\left\vert
f\right\vert \right) dt_{1}dt_{2}\right) ^{2}\right) ^{1/2}.
\end{equation*}%
Since%
\begin{equation*}
\sum\limits_{l_{1}=0}^{2^{m_{1}}-1}\sum\limits_{l_{2}=0}^{2^{m_{2}}-1}\left%
\vert P_{m_{1}m_{2}}^{\left( 1\right) }\left( \frac{l_{1}}{2^{m_{1}}},\frac{%
l_{2}}{2^{m_{2}}}\right) \right\vert ^{2}
\end{equation*}%
\begin{equation*}
=2^{m_{1}+m_{2}}\sum\limits_{l_{1}=0}^{2^{m_{1}}-1}\sum%
\limits_{l_{2}=0}^{2^{m_{2}}-1}\int\limits_{l_{1}2^{-m_{1}}}^{\left(
l_{1}+1\right) 2^{-m_{1}}}\int\limits_{l_{2}2^{-m_{2}}}^{\left(
l_{2}+1\right) 2^{-m_{2}}}\left\vert P_{m_{1}m_{2}}^{\left( 1\right) }\left(
t_{1},t_{2}\right) \right\vert ^{2}dt_{1}dt_{2}
\end{equation*}%
\begin{equation*}
2^{m_{1}+m_{2}}\int\limits_{\mathbb{I}^{2}}\left\vert P_{m_{1}m_{2}}^{\left(
1\right) }\left( t_{1},t_{2}\right) \right\vert ^{2}dt_{1}dt_{2}\leq
2^{m_{1}+m_{2}}
\end{equation*}%
we have%
\begin{equation}
R_{1}\lesssim V_{m_{1}m_{2}}\left( x_{1},x_{2};\left\vert f\right\vert
\right)  \label{R1-2}
\end{equation}

Analogously, from (\ref{R2})-(\ref{R9}) we can prove that ($s=2,...,9$)%
\begin{equation}
R_{s}\lesssim V_{m_{1}m_{2}}\left( x_{1},x_{2};\left\vert f\right\vert
\right) .  \label{R2-R9}
\end{equation}

Combining (\ref{R1-R9}), (\ref{R1-2}) and (\ref{R2-R9}) we conclude the
proof of Theorem \ref{2-estimation}.
\end{proof}

\begin{lemma}
\label{iteration}The following estimation holds%
\begin{equation*}
V_{m_{1}m_{2}}\left( x_{1},x_{2};\left\vert f\right\vert \right) \lesssim
V^{\left( 1\right) }\left( x_{1},x_{2};V^{\left( 2\right) }\left( \left\vert
f\right\vert \right) \right) .
\end{equation*}
\end{lemma}

\begin{proof}[Proof of Lemma \protect\ref{iteration}]
There is a suitable vector%
\begin{equation*}
\left\{ a_{l_{1}l_{2}}\left( x_{1},x_{2}\right) :0\leq l_{1}<2^{m_{1}},0\leq
l_{2}<2^{m_{2}}\right\}
\end{equation*}%
such that%
\begin{equation*}
\sum\limits_{l_{1}=0}^{2^{m_{1}}-1}\sum\limits_{l_{2}=0}^{2^{m_{2}}-1}\left%
\vert a_{l_{1}l_{2}}\left( x_{1},x_{2}\right) \right\vert ^{2}=1
\end{equation*}%
and%
\begin{equation*}
V_{m_{1}m_{2}}\left( x_{1},x_{2};\left\vert f\right\vert \right)
\end{equation*}%
\begin{equation*}
=\sum\limits_{l_{1}=0}^{2^{m_{1}}-1}\int\limits_{l_{1}2^{-m_{1}}}^{\left(
l_{1}+1\right) 2^{-m_{1}}}\sum\limits_{j_{1}=0}^{m_{1}-1}2^{j_{1}-1}\mathbb{I%
}_{I_{j_{1}}}\left( t_{1}\right)
\end{equation*}%
\begin{equation*}
\times \sum\limits_{l_{2}=0}^{2^{m_{2}}-1}a_{l_{1}l_{2}}\left(
x_{1},x_{2}\right) \left( \int\limits_{l_{2}2^{-m_{2}}}^{\left(
l_{2}+1\right) 2^{-m_{2}}}\sum\limits_{j_{2}=0}^{m_{2}-1}2^{j_{2}-1}\mathbb{I%
}_{I_{j_{2}}}\left( t_{2}\right) \right.
\end{equation*}%
\begin{equation*}
\left. \times S_{2^{m_{1}}2^{m_{2}}}\left(
x_{1}+t_{1}+e_{j_{1}},x_{2}+t_{2}+e_{j_{2}};\left\vert f\right\vert \right)
dt_{2}\right) dt_{1}
\end{equation*}%
\begin{equation*}
\leq
\sum\limits_{l_{1}=0}^{2^{m_{1}}-1}\int\limits_{l_{1}2^{-m_{1}}}^{\left(
l_{1}+1\right) 2^{-m_{1}}}\sum\limits_{j_{1}=0}^{m_{1}-1}2^{j_{1}-1}\mathbb{I%
}_{I_{j_{1}}}\left( t_{1}\right) \left(
\sum\limits_{l_{2}=0}^{2^{m_{2}}-1}\left\vert a_{l_{1}l_{2}}\left(
x_{1},x_{2}\right) \right\vert ^{2}\right)
^{1/2}\sum\limits_{l_{2}=0}^{2^{m_{2}}-1}
\end{equation*}%
\begin{equation*}
\left( \int\limits_{l_{2}2^{-m_{2}}}^{\left( l_{2}+1\right)
2^{-m_{2}}}\sum\limits_{j_{2}=0}^{m_{2}-1}2^{j_{2}-1}\mathbb{I}%
_{I_{j_{2}}}\left( t_{2}\right) \right.
\end{equation*}%
\begin{equation*}
\left. \left. \times S_{2^{m_{2}}}^{\left( 2\right) }\left( x_{1}\dotplus
t_{1}\dotplus e_{j_{1}},x_{2}\dotplus t_{2}\dotplus e_{j_{2}};\left\vert
f\right\vert \right) dt_{2}\right) ^{2}\right) ^{1/2}dt_{1}
\end{equation*}%
\begin{equation*}
\leq \sum\limits_{l_{1}=0}^{2^{m_{1}}-1}\left(
\sum\limits_{l_{2}=0}^{2^{m_{2}}-1}\left\vert a_{l_{1}l_{2}}\left(
x_{1},x_{2}\right) \right\vert ^{2}\right)
^{1/2}\int\limits_{l_{1}2^{-m_{1}}}^{\left( l_{1}+1\right)
2^{-m_{1}}}\sum\limits_{j_{1}=0}^{m_{1}-1}2^{j_{1}-1}\mathbb{I}%
_{I_{j_{1}}}\left( t_{1}\right)
\end{equation*}%
\begin{equation*}
\times V^{\left( 2\right) }\left( x_{1}+t_{1}+e_{j_{1}},x_{2};\left\vert
f\right\vert \right) dt_{1}
\end{equation*}%
\begin{equation*}
\leq \left(
\sum\limits_{l_{1}=0}^{2^{m_{1}}-1}\sum\limits_{l_{2}=0}^{2^{m_{2}}-1}\left%
\vert a_{l_{1}l_{2}}\left( x_{1},x_{2}\right) \right\vert ^{2}\right)
^{1/2}\left( \sum\limits_{l_{1}=0}^{2^{m_{1}}-1}\left(
\int\limits_{l_{1}2^{-m_{1}}}^{\left( l_{1}+1\right)
2^{-m_{1}}}\sum\limits_{j_{1}=0}^{m_{1}-1}2^{j_{1}-1}\mathbb{I}%
_{I_{j_{1}}}\left( t_{1}\right) \right. \right.
\end{equation*}%
\begin{equation*}
\left. \left. \times V^{\left( 2\right) }\left( x_{1}\dotplus t_{1}\dotplus
e_{j_{1}},x_{2};\left\vert f\right\vert \right) dt_{1}\right) ^{2}\right)
^{1/2}\lesssim V^{\left( 1\right) }\left( x_{1},x_{2};V^{\left( 2\right)
}\right) .
\end{equation*}

Lemma \ref{iteration} is proved.
\end{proof}

\section{Proof of Main Results}

\begin{proof}[Proof of Theorem \protect\ref{bmo}]
Set%
\begin{equation*}
f_{n_{1},n_{2}}\left( x_{1},x_{2},t_{1},t_{2}\right)
:=\sum\limits_{k_{1}=0}^{2^{n_{1}}-1}\sum%
\limits_{k_{2}=0}^{2^{n_{2}}-1}S_{k_{1},k_{2}}\left( x_{1},x_{2};f\right) 
\mathbb{I}_{\delta _{k_{1}}^{n_{1}}}\left( t_{1}\right) \mathbb{I}_{\delta
_{k_{2}}^{n_{2}}}\left( t_{2}\right) ,
\end{equation*}%
\begin{equation*}
J_{1}:=\left[ j_{1}2^{-m},\left( j_{1}+1\right) 2^{-m}\right) ,J_{2}:=\left[
j_{2}2^{-m},\left( j_{2}+1\right) 2^{-m}\right) .
\end{equation*}

Then we can write $\left( n_{1}\leq n_{2}\right) $%
\begin{equation}
\left\Vert f_{n_{1},n_{2}}\left( x_{1},x_{2},\cdot ,\cdot \right)
\right\Vert _{BMO}  \label{P1-P4}
\end{equation}%
\begin{equation*}
=\sup\limits_{m}\sup\limits_{0\leq j_{1},j_{2}<2^{m}}\left( \frac{1}{%
\left\vert J_{1}\times J_{2}\right\vert }\int\limits_{J_{1}\times
J_{2}}\left\vert f_{n_{1},n_{2}}\left( x_{1},x_{2},t_{1},t_{2}\right)
\right. \right.
\end{equation*}%
\begin{equation*}
\left. \left. -\frac{1}{\left\vert J_{1}\times J_{2}\right\vert }%
\int\limits_{J_{1}\times J_{2}}f_{n_{1},n_{2}}\left(
x_{1},x_{2},u_{1},u_{2}\right) du_{1}du_{2}\right\vert
^{2}dt_{1}dt_{2}\right) ^{1/2}
\end{equation*}%
\begin{equation*}
\leq \left( \sup\limits_{m\leq n_{1}}\sup\limits_{0\leq
j_{1},j_{2}<2^{m}}+\sup\limits_{n_{1}<m\leq n_{2}}\sup\limits_{0\leq
j_{1},j_{2}<2^{m}}+\sup\limits_{m>n_{2}}\sup\limits_{0\leq
j_{1},j_{2}<2^{m}}\right)
\end{equation*}%
\begin{equation*}
\left( \frac{1}{\left\vert J_{1}\times J_{2}\right\vert }\int\limits_{J_{1}%
\times J_{2}}\left\vert f_{n_{1},n_{2}}\left( x_{1},x_{2},t_{1},t_{2}\right)
\right. \right.
\end{equation*}%
\begin{equation*}
\left. \left. -\frac{1}{\left\vert J_{1}\times J_{2}\right\vert }%
\int\limits_{J_{1}\times J_{2}}f_{n_{1},n_{2}}\left(
x_{1},x_{2},u_{1},u_{2}\right) du_{1}du_{2}\right\vert
^{2}dt_{1}dt_{2}\right) ^{1/2}
\end{equation*}%
\begin{equation*}
:=P_{1}\left( n_{1},n_{2}\right) +P_{2}\left( n_{1},n_{2}\right)
+P_{3}\left( n_{1},n_{2}\right) .
\end{equation*}

et $n_{1}\leq n_{2}<m$. Since $f_{n_{1},n_{2}}\left(
x_{1},x_{2},t_{1},t_{2}\right) $ is constant on$\ \left[ \frac{j_{1}}{2^{m}},%
\frac{j_{1}+1}{2^{m}}\right) \times \left[ \frac{j_{2}}{2^{m}},\frac{j_{2}+1%
}{2^{m}}\right) $ for fixed $\left( x_{1},x_{2}\right) \in \mathbb{I}^{2}$
we conclude that

\begin{equation}
P_{3}\left( n_{1},n_{2}\right) =0.  \label{P4}
\end{equation}

Let $m\leq n_{1}$. Then for $P_{1}$ we can write%
\begin{equation*}
P_{1}\left( n_{1},n_{2}\right)
\end{equation*}%
\begin{equation*}
=\sup\limits_{m\leq n_{1}}\sup\limits_{0\leq j_{1},j_{2}<2^{m}}\left(
2^{2m}\int\limits_{J_{1}\times J_{2}}\left\vert
\sum\limits_{k_{1}=j_{1}2^{n_{1}-m}}^{\left( j_{1}+1\right)
2^{n_{1}-m}-1}\sum\limits_{k_{2}=j_{2}2^{n_{2}-m}}^{\left( j_{2}+1\right)
2^{n_{2}-m}-1}\right. \right.
\end{equation*}%
\begin{equation*}
S_{k_{1},k_{2}}\left( x_{1},x_{2};f\right) \mathbb{I}_{\delta
_{k_{1}}^{n_{1}}}\left( t_{1}\right) \mathbb{I}_{\delta
_{k_{2}}^{n_{2}}}\left( t_{2}\right) -2^{2m}
\end{equation*}%
\begin{equation*}
\times \int\limits_{J_{1}\times
J_{2}}\sum\limits_{k_{1}=j_{1}2^{n_{1}-m}}^{\left( j_{1}+1\right)
2^{n_{1}-m}-1}\sum\limits_{k_{2}=j_{2}2^{n_{2}-m}}^{\left( j_{2}+1\right)
2^{n_{2}-m}-1}S_{k_{1},k_{2}}\left( x_{1},x_{2};f\right)
\end{equation*}%
\begin{equation*}
\left. \left. \times \mathbb{I}_{\delta _{k_{1}}^{n_{1}}}\left( u_{1}\right) 
\mathbb{I}_{\delta _{k_{2}}^{n_{2}}}\left( u_{2}\right)
du_{1}du_{2}\right\vert ^{2}dt_{1}dt_{2}\right) ^{1/2}
\end{equation*}%
\begin{equation*}
=\sup\limits_{m\leq n_{1}}\sup\limits_{0\leq j_{1},j_{2}<2^{m}}\left(
2^{m-n_{1}}2^{m-n_{2}}\sum\limits_{k_{1}=j_{1}2^{n_{1}-m}}^{\left(
j_{1}+1\right) 2^{n_{1}-m}-1}\sum\limits_{k_{2}=j_{2}2^{n_{2}-m}}^{\left(
j_{2}+1\right) 2^{n_{2}-m}-1}\right.
\end{equation*}%
\begin{equation*}
\left\vert S_{k_{1},k_{2}}\left( x_{1},x_{2};f\right)
-2^{m-n_{1}}2^{m-n_{2}}\right.
\end{equation*}%
\begin{equation*}
\left. \left. \times \sum\limits_{k_{1}=j_{1}2^{n_{1}-m}}^{\left(
j_{1}+1\right) 2^{n_{1}-m}-1}\sum\limits_{k_{2}=j_{2}2^{n_{2}-m}}^{\left(
j_{2}+1\right) 2^{n_{2}-m}-1}S_{k_{1},k_{2}}\left( x_{1},x_{2};f\right)
\right\vert ^{2}\right) ^{1/2}
\end{equation*}

\begin{equation*}
=\sup\limits_{m_{1}\leq n_{1},m_{2}\leq n_{2}}\sup\limits_{0\leq
j_{1}<2^{m_{1}},0\leq j_{2}<2^{m_{2}}}\left(
2^{-m_{1}}2^{-m_{2}}\sum\limits_{l_{1}=0}^{2^{m_{1}}-1}\sum%
\limits_{l_{2}=0}^{2^{m_{2}}-1}\right.
\end{equation*}%
\begin{equation*}
\left\vert S_{l_{1}+j_{1}2^{m_{1}},l_{2}+j_{2}2^{m_{2}}}\left(
x_{1},x_{2};f\right) -2^{-m_{1}}2^{-m_{2}}\right.
\end{equation*}%
\begin{equation*}
\left. \left. \times
\sum\limits_{q_{1}=0}^{2^{m_{1}}-1}\sum%
\limits_{q_{2}=0}^{2^{m_{2}}-1}S_{q_{1}+j_{1}2^{m_{1}},q_{2}+j_{2}2^{m_{2}}}%
\left( x_{1},x_{2};f\right) \right\vert ^{2}\right) ^{1/2}.
\end{equation*}

Since%
\begin{eqnarray*}
&&S_{l_{1}+j_{1}2^{m_{1}},l_{2}+j_{2}2^{m_{2}}}\left( x_{1},x_{2};f\right) \\
&=&S_{j_{1}2^{m_{1}},j_{2}2^{m_{2}}}\left( x_{1},x_{2};f\right)
+S_{l_{1},j_{2}2^{m_{2}}}\left( x_{1},x_{2};fw_{j_{1}2^{m_{1}}}\right)
w_{j_{1}2^{m_{1}}}\left( x_{1}\right) \\
&&+S_{j_{1}2^{m_{1}},l_{2}}\left( x_{1},x_{2};fw_{j_{2}2^{m_{2}}}\right)
w_{j_{2}2^{m_{2}}}\left( x_{2}\right) \\
&&+S_{l_{1},l_{2}}\left( x_{1},x_{2};fw_{j_{1}2^{m_{1}}}\otimes
w_{j_{2}2^{m_{2}}}\right) w_{j_{1}2^{m_{1}}}\left( x_{1}\right)
w_{j_{2}2^{m_{2}}}\left( x_{2}\right)
\end{eqnarray*}%
we can write%
\begin{equation}
P_{1}\left( n_{1},n_{2}\right) \leq \sup\limits_{m_{1}\leq n_{1},m_{2}\leq
n_{2}}\sup\limits_{0\leq j_{1}<2^{m_{1}},0\leq j_{2}<2^{m_{2}}}\left(
2^{-m_{1}-m_{2}}\sum\limits_{l_{1}=0}^{2^{m_{1}}-1}\sum%
\limits_{l_{2}=0}^{2^{m_{2}}-1}\right.  \label{P11-P13}
\end{equation}%
\begin{equation*}
S_{l_{1},l_{2}}\left( x_{1},x_{2};fw_{j_{1}2^{m_{1}}}\otimes
w_{j_{2}2^{m_{2}}}\right) -2^{-m_{1}-m_{2}}
\end{equation*}%
\begin{equation*}
-\left. \left. \times
\sum\limits_{q_{1}=0}^{2^{m_{1}}-1}\sum%
\limits_{q_{2}=0}^{2^{m_{2}}-1}S_{q_{1},q_{2}}\left(
x_{1},x_{2};fw_{j_{1}2^{m_{1}}}\otimes w_{j_{2}2^{m_{2}}}\right) \right\vert
^{2}\right) ^{1/2}
\end{equation*}%
\begin{equation*}
+\sup\limits_{m_{1}\leq n_{1},m_{2}\leq n_{2}}\sup\limits_{0\leq
j_{1}<2^{m_{1}},0\leq j_{2}<2^{m_{2}}}\left(
2^{-m_{1}}\sum\limits_{l_{1}=0}^{2^{m_{1}}-1}\left\vert
S_{l_{1},j_{2}2^{m_{2}}}\left( x_{1},x_{2};fw_{j_{1}2^{m_{1}}}\right)
\right. \right.
\end{equation*}%
\begin{equation*}
-\left. \left.
2^{-m_{1}}\sum\limits_{q_{1}=0}^{2^{m_{1}}-1}S_{q_{1},j_{2}2^{m_{2}}}\left(
x_{1},x_{2};fw_{j_{1}2^{m_{1}}}\right) \right\vert ^{2}\right) ^{1/2}
\end{equation*}%
\begin{equation*}
+\sup\limits_{m_{1}\leq n_{1},m_{2}\leq n_{2}}\sup\limits_{0\leq
j_{1}<2^{m_{1}},0\leq j_{2}<2^{m_{2}}}\left(
2^{-m_{2}}\sum\limits_{l_{2}=0}^{2^{m_{2}}-1}\left\vert
S_{j_{1}2^{m_{1}},l_{2}}\left( x_{1},x_{2};fw_{j_{2}2^{m_{2}}}\right)
\right. \right.
\end{equation*}%
\begin{equation*}
-\left. \left.
2^{-m_{1}}\sum\limits_{q_{2}=0}^{2^{m_{2}}-1}S_{j_{1}2^{m_{1}},q_{2}}\left(
x_{1},x_{2};fw_{j_{2}2^{m_{2}}}\right) \right\vert ^{2}\right) ^{1/2}
\end{equation*}%
\begin{equation*}
:=P_{11}\left( n_{1},n_{2}\right) +P_{12}\left( n_{1},n_{2}\right)
+P_{13}\left( n_{1},n_{2}\right) .
\end{equation*}

From Lemmas \ref{2-estimation} and \ref{iteration} we obtain%
\begin{equation*}
P_{11}\left( n_{1},n_{2}\right) \lesssim \sup\limits_{m_{1}\leq
n_{1},m_{2}\leq n_{2}}\sup\limits_{0\leq j_{1}<2^{m_{1}},0\leq
j_{2}<2^{m_{2}}}V_{m_{1},m_{2}}\left( x_{1},x_{2};\left\vert
fw_{j_{1}2^{m_{1}}}\otimes w_{j_{2}2^{m_{2}}}\right\vert \right) 
\end{equation*}%
\begin{equation*}
\lesssim V^{\left( 1\right) }\left( x_{1},x_{2};V^{\left( 2\right) }\left(
\left\vert f\right\vert \right) \right) .
\end{equation*}%
Consequently, by Theorem Sch 1 we can write%
\begin{equation}
\left\vert \left\{ \left( x_{1},x_{2}\right) \in \mathbb{I}%
^{2}:\sup\limits_{n_{1},n_{2}}P_{11}\left( n_{1},n_{2}\right) >\lambda
\right\} \right\vert   \label{P11}
\end{equation}%
\begin{equation*}
\lesssim \frac{1}{\lambda }\int\limits_{\mathbb{I}}\left( \int\limits_{%
\mathbb{I}}V^{\left( 2\right) }\left( x_{1},x_{2};\left\vert f\right\vert
\right) dx_{1}\right) dx_{2}
\end{equation*}%
\begin{equation*}
\lesssim \frac{1}{\lambda }\int\limits_{\mathbb{I}}\left( \int\limits_{%
\mathbb{I}}\left\vert f\left( x_{1},x_{2}\right) \right\vert \log
^{+}\left\vert f\left( x_{1},x_{2}\right) \right\vert dx_{2}+1\right) dx_{1}
\end{equation*}%
\begin{equation*}
\lesssim \frac{1}{\lambda }\left( \int\limits_{\mathbb{I}^{2}}\left\vert
f\right\vert \log ^{+}\left\vert f\right\vert +1\right) .
\end{equation*}

Since%
\begin{equation*}
S_{l_{1},j_{2}2^{m}}\left( x_{1},x_{2};fw_{j_{1}2^{m}}\right)
=S_{l_{1}}^{\left( 1\right) }\left( x_{1},x_{2};S_{j_{2}2^{m}}\left(
f\right) w_{j_{1}2^{m}}\right) 
\end{equation*}%
from Lemma \ref{SchD2} we have%
\begin{equation*}
P_{12}\left( n_{1},n_{2}\right) 
\end{equation*}%
\begin{equation*}
\lesssim \sup\limits_{m_{1}\leq n_{1},m_{2}\leq n_{2}}\sup\limits_{0\leq
j_{1}<2^{m_{1}},0\leq j_{2}<2^{m_{2}}}\left(
2^{-m_{1}}\sum\limits_{l_{1}=0}^{2^{m_{1}}-1}\left\vert S_{l_{1}}^{\left(
1\right) }\left( x_{1},x_{2};S_{j_{2}2^{m_{2}}}\left( f\right)
w_{j_{1}2^{m_{1}}}\right) \right\vert ^{2}\right) ^{1/2}
\end{equation*}%
\begin{equation*}
\leq \sup\limits_{m_{1}\leq n_{1},m_{2}\leq n_{2}}\sup\limits_{0\leq
j_{1}<2^{m_{1}},0\leq j_{2}<2^{m_{2}}}V^{\left( 1\right) }\left(
x_{1},x_{2};\left\vert S_{j_{2}2^{m_{2}}}^{\left( 2\right) }\left( f\right)
w_{j_{1}2^{m_{1}}}\right\vert \right) 
\end{equation*}%
\begin{equation*}
\leq V^{\left( 1\right) }\left( x_{1},x_{2};S_{\ast }^{\left( 2\right)
}\left( f\right) \right) ,
\end{equation*}%
where%
\begin{equation*}
S_{\ast }^{\left( 2\right) }\left( f\right) :=\sup\limits_{n}\left\vert
S_{n}^{\left( 2\right) }\left( f\right) \right\vert .
\end{equation*}

If $f\in L\left( \log ^{+}L\right) ^{2}\left( \mathbb{I}^{2}\right) $. Then $%
f\left( x_{1},\cdot \right) \in L\left( \log ^{+}L\right) ^{2}\left( \mathbb{%
I}\right) $ for a. e. $x_{1}\in \mathbb{I}$, and from the well-known theorem
(see \cite{tateoka}) $S_{\ast }^{\left( 2\right) }\left( x_{1},\cdot
,f\right) \in L_{1}\left( \mathbb{I}\right) $ for a. e. $x_{1}\in \mathbb{I}$%
. Moreover%
\begin{equation*}
\int\limits_{\mathbb{I}}S_{\ast }^{\left( 2\right) }\left(
x_{1},x_{2};f\right) dx_{2}\lesssim \int\limits_{\mathbb{I}}\left\vert
f\left( x_{1},x_{2}\right) \right\vert \left( \log ^{+}\left\vert f\left(
x_{1},x_{2}\right) \right\vert \right) ^{2}dx_{2}+1
\end{equation*}%
for a. e. $x_{1}\in \mathbb{I}$.

Hence,%
\begin{equation}
\left\vert \left\{ \left( x_{1},x_{2}\right) \in \mathbb{I}%
^{2}:\sup\limits_{n_{1},n_{2}}P_{12}\left( n_{1},n_{2}\right) >\lambda
\right\} \right\vert  \label{P12}
\end{equation}%
\begin{equation*}
\lesssim \left\vert \left\{ \left( x_{1},x_{2}\right) \in \mathbb{I}%
^{2}:V^{\left( 1\right) }\left( x_{1},x_{2};S_{\ast }^{\left( 2\right)
}\left( f\right) \right) >\lambda \right\} \right\vert
\end{equation*}%
\begin{equation*}
\lesssim \frac{1}{\lambda }\int\limits_{\mathbb{I}}\left( \int\limits_{%
\mathbb{I}}S_{\ast }^{\left( 2\right) }\left( x_{1},x_{2};f\right)
dx_{1}\right) dx_{2}
\end{equation*}%
\begin{equation*}
\lesssim \frac{1}{\lambda }\int\limits_{\mathbb{I}}\left( \int\limits_{%
\mathbb{I}}\left\vert f\left( x_{1},x_{2}\right) \right\vert \left( \log
^{+}\left\vert f\left( x_{1},x_{2}\right) \right\vert \right)
^{2}dx_{2}+1\right) dx_{1}
\end{equation*}%
\begin{equation*}
\lesssim \frac{1}{\lambda }\left( \int\limits_{\mathbb{I}^{2}}\left\vert
f\right\vert \left( \log ^{+}\left\vert f\right\vert \right) ^{2}+1\right) .
\end{equation*}

Analogously, we can prove that%
\begin{equation}
\left\vert \left\{ \left( x_{1},x_{2}\right) \in \mathbb{I}%
^{2}:\sup\limits_{n_{1},n_{2}}P_{13}\left( n_{1},n_{2}\right) >\lambda
\right\} \right\vert   \label{P13}
\end{equation}%
\begin{equation*}
\lesssim \frac{1}{\lambda }\left( \int\limits_{\mathbb{I}^{2}}\left\vert
f\right\vert \left( \log ^{+}\left\vert f\right\vert \right) ^{2}+1\right) .
\end{equation*}

Combining (\ref{P11-P13})- (\ref{P13}) we get%
\begin{equation}
\left\vert \left\{ \left( x_{1},x_{2}\right) \in \mathbb{I}%
^{2}:\sup\limits_{n_{1},n_{2}}P_{1}\left( n_{1},n_{2}\right) >\lambda
\right\} \right\vert  \label{P1}
\end{equation}%
\begin{equation*}
\lesssim \frac{1}{\lambda }\left( \int\limits_{\mathbb{I}^{2}}\left\vert
f\right\vert \left( \log ^{+}\left\vert f\right\vert \right) ^{2}+1\right) .
\end{equation*}

Analogously, we can prove that%
\begin{equation}
\left\vert \left\{ \left( x_{1},x_{2}\right) \in \mathbb{I}%
^{2}:\sup\limits_{n_{1},n_{2}}P_{2}\left( n_{1},n_{2}\right) >\lambda
\right\} \right\vert  \label{P2-P3}
\end{equation}%
\begin{equation*}
\lesssim \frac{1}{\lambda }\left( \int\limits_{\mathbb{I}^{2}}\left\vert
f\right\vert \left( \log ^{+}\left\vert f\right\vert \right) ^{2}+1\right) .
\end{equation*}

Combining (\ref{P1-P4}), (\ref{P4}), (\ref{P1}) and (\ref{P2-P3}) we
complete the proof of Theorem \ref{bmo}.
\end{proof}

\end{document}